\documentclass[11pt,letterpaper,reqno]{amsart}
\usepackage{tikz}
\usetikzlibrary{positioning, shapes.geometric, arrows.meta, calc}
\usepackage{amssymb}
\usepackage{amsmath}
\usepackage{amsthm}
\usepackage{amsfonts}
\usepackage{bbm}
\usepackage{enumitem} 
\usepackage{pgfplots}
\pgfplotsset{compat=1.18} 
\usepackage{booktabs}
\usepackage{graphicx}
\usepackage[T1]{fontenc}
\usepackage{doi}
\usepackage{xcolor}
\usepackage{float} 
\usepackage{listings}
\usepackage{bookmark}
\usepackage{hyperref}
\usepackage{comment}
\hypersetup{pdfstartview={FitH}}

\newtheorem{thm}{Theorem}[section]
\newtheorem{lem}[thm]{Lemma}

\theoremstyle{definition}

\numberwithin{equation}{section}

\newcommand{\seqnum}[1]{\href{https://oeis.org/#1}{\rm \underline{#1}}}
\makeatother

\begin{document}

\title[Smallest missing denominator]{The smallest denominator not contained in a unit fraction decomposition of $1$ with fixed length}

\author[Wouter van Doorn]{Wouter van Doorn}
\address{Wouter van Doorn, Groningen, the Netherlands} 
\email{wonterman1@hotmail.com}

\author[Quanyu Tang]{Quanyu Tang}
\address{School of Mathematics and Statistics, Xi'an Jiaotong University, Xi'an 710049, P. R. China}
\email{tang\_quanyu@163.com}


\begin{abstract}
Let $v(k)$ be the smallest integer larger than $1$ that does not occur among the denominators in any identity of the form
$$
1=\frac1{n_1}+\cdots+\frac1{n_k},
$$
where $1 \le n_1<\cdots<n_k$ are pairwise distinct integers. In their 1980 monograph, Erd\H{o}s and Graham asked for quantitative estimates on the growth of $v(k)$ and suggested the lower bound $v(k)\gg k!$. In this paper we give the first known improvement and show that there exists an absolute constant $c>0$ such that the inequality
$$
v(k)\ge e^{c k^2}
$$
holds for all positive integers $k$.
\end{abstract}

\maketitle

\section{Introduction}\label{sec:intro}
For a positive integer $k$, let $S_k$ denote the set of all $k$-tuples
$(n_1,\dots,n_k)$ of integers satisfying
\begin{equation}\label{eq:ufdowt_1}
1=\frac1{n_1}+\frac1{n_2}+\cdots+\frac1{n_k},
\end{equation}
with $1\le n_1<n_2<\cdots<n_k$. We call an equation of the form \eqref{eq:ufdowt_1} a \emph{\(k\)-term unit fraction decomposition of \(1\)}. Such sums are also known in the literature as~\emph{Egyptian fraction expansions} or simply~\emph{Egyptian fractions}. In some works,
repetitions among the denominators are allowed. Throughout this paper, however, we only consider decompositions with pairwise distinct denominators. We further define
\[
F(k):=|S_k|
\]
and 
\[
D_k:=\{m:\ m=n_i \text{ for some }(n_1,\dots,n_k)\in S_k\text{ and some }i \text{ with }1 \le i \le k\}.
\]
Thus, $D_k$ is the set of all denominators that occur in at least one $k$-term unit fraction decomposition of $1$.

In their 1980 monograph~\cite[p.~35]{ErGr80}, Erd\H{o}s and Graham defined $v(k)$ to be the smallest integer $>1$ not contained in $D_k$, and asked how fast $v(k)$ grows as a function of $k$. They remarked that the lower bound $v(k)\gg k!$ should
follow from earlier work of Bleicher and Erd\H{o}s~\cite{BlEr1,BlEr2,BlEr3}, and further speculated that $v(k)$ might even grow doubly exponentially. Graham reiterated the same question in \cite[p.~297]{Gra13}, and it is now listed as Problem~\#293 on Bloom's Erd\H{o}s Problems website~\cite{EP}.

For an upper bound, note that we certainly have 
\begin{equation} \label{eq:vkupper}
    v(k) \le |D_k| + 2 \le kF(k) + 2.
\end{equation} Let $c_0=1.264085\ldots$ denote the Vardi constant.\footnote{See OEIS sequence \seqnum{A076393} for its decimal expansion~\cite{OEIS}.} By applying the upper bound on $F(k)$ due to Elsholtz and Planitzer~\cite[Corollary~3]{ElPl21}, we then obtain
\[
v(k)\le c_0^{\left(\frac15+o(1)\right)2^k},
\]
as the prefactor $k$ and the additive constant $+2$ on the right-hand side of \eqref{eq:vkupper} can both be absorbed into the $o(1)$-term.

On the other hand, to the best of our knowledge, no lower bound on $v(k)$ exists in the literature. In fact, even extracting the lower bound $v(k)\gg k!$ from the results in \cite{BlEr1}, \cite{BlEr2} and \cite{BlEr3} (which Erd\H{o}s and Graham in \cite{ErGr80} and \cite{Gra13} claim is `easy to see'), does not seem straightforward to us. Regardless, the purpose of this note is to prove the following stronger bound.

\begin{thm}\label{thm:main_vkeck2}
There exists an absolute constant $c>0$ such that
\[
v(k)\ge e^{ck^2}
\]
holds for all positive integers $k$.
\end{thm}

\section{Main lower bound}\label{sec:main_lower_bound}
As observed by Konyagin~\cite{Ko14}, we have $F(k) \le F(k+1)$ for all $k \ge 2$, due to the identity
\begin{equation} \label{eq:basic}
\frac{1}{n} = \frac{1}{n+1} + \frac{1}{n(n+1)}.
\end{equation}
Indeed, \eqref{eq:basic} provides an injection from $S_k$ to $S_{k+1}$, sending $$(n_1, \ldots, n_k) \in S_k$$ to $$(n_1, \ldots, n_{k-1}, n_k+1, n_k(n_k+1)) \in S_{k+1}.$$ By a slightly more involved argument we can even show the inclusion $D_k \subseteq D_{k+1}$.

\begin{lem}\label{lem:extension}
For all $k \ge 2$ we have $D_k \subseteq D_{k+1}$.
\end{lem}
\begin{proof}
Let $m$ be an element of $D_k$ and let $(n_1, \ldots, n_k) \in S_k$ contain $m$. We then aim to construct a $(k+1)$-tuple in $S_{k+1}$ that also contains $m$. As one can verify $$D_2 = \emptyset \subseteq D_3 = \{2, 3, 6\} \subseteq \{2, 3, 7, 42\} \cup \{2, 4, 6, 12\} \subseteq D_4,$$ we may assume $k\ge 4$. 

In analogy with the proof of $F(k) \le F(k+1)$, if $m \neq n_k$, then we first of all remark $$m \in (n_1,\dots,n_{k-1},n_k+1,n_k(n_k+1)) \in S_{k+1}$$ by \eqref{eq:basic}. Secondly, if $m = n_k \notin \{n_{k-1}+1,\;n_{k-1}(n_{k-1}+1)\}$, then we have\footnote{Here, as in the rest of this proof, we slightly abuse notation by writing that a $(k+1)$-tuple is contained in $S_{k+1}$, where it is actually the sorted tuple that is contained in $S_{k+1}$.} $$m \in (n_1,\dots,n_{k-2}, n_k, n_{k-1}+1,n_{k-1}(n_{k-1}+1)) \in S_{k+1},$$ once again by \eqref{eq:basic}. We may therefore further assume $m = n_k \in \{n_{k-1}+1,\;n_{k-1}(n_{k-1}+1)\}$ from now on.

If $n_{k-1}$ is composite, write $n_{k-1} = ab$ with $a$ and $b$ integers larger than $1$. Generalizing~\eqref{eq:basic}, we have the identity $$\frac{1}{n_{k-1}}=\frac{1}{n_{k-1}+a}+\frac{1}{b(n_{k-1}+a)},$$ while $$n_{k-1}+1 < n_{k-1}+a < b(n_{k-1}+a) < n_{k-1}(n_{k-1}+1).$$ We therefore deduce $$m \in (n_1,\dots,n_{k-2}, n_k, n_{k-1}+a, b(n_{k-1}+a)) \in S_{k+1}.$$ 

On the other hand, if $n_{k-1} = p$ is prime, we claim that $m = n_k \in \{n_{k-1}+1,\;n_{k-1}(n_{k-1}+1)\}$ implies $n_k = n_{k-1}(n_{k-1}+1)$. To see this, assume by contradiction that $m = n_k$ is equal to $n_{k-1}+1$ instead. We then write
\[
\sum_{i=1}^k \frac{1}{n_i}
=\frac{1}{p}+\sum_{\substack{1\le i\le k\\ i\neq k-1}}\frac{1}{n_i}
=\frac{1}{p}+\frac{A}{B}=\frac{Ap+B}{Bp},
\]
where $\frac{A}{B}$ is the reduced form of the remaining sum.
Now, by the assumption $n_k = n_{k-1}+1$, we see that none of the denominators $n_i$ with $i\neq k-1$ are divisible by $p$,
hence $\gcd(B,p)=1$. Thus, for the numerator we have $Ap+B\equiv B\not\equiv 0\pmod p$, implying that the fraction $(Ap+B)/(Bp)$ cannot simplify
to a fraction with denominator not divisible by $p$. In particular, this fraction is not equal to
the integer $1$, contradicting the assumption $(n_1,\ldots,n_k)\in S_k$.

With the assumption that $n_{k-1}$ is prime and the equality $n_k = n_{k-1}(n_{k-1}+1)$ in mind, there are now two final cases to consider.

If $n_{k-1} \neq n_{k-2} + 1$, then we claim that $$m \in (n_1,\dots,n_{k-3},n_{k-1},n_k,n_{k-2}+1,n_{k-2}(n_{k-2}+1)) \in S_{k+1}.$$ The sum of reciprocals of this latter tuple is equal to $1$ by~\eqref{eq:basic}, so it suffices to prove that all integers in the tuple are distinct. In other words, we need to show that $$\{n_{k-1}, n_k\} \cap \{n_{k-2}+1,n_{k-2}(n_{k-2}+1) \} = \emptyset.$$ One sees this by the assumption $n_{k-1} \neq n_{k-2} + 1$, the assumption that $n_{k-1}$ is prime (hence not equal to $n_{k-2}(n_{k-2}+1)$), and the fact that $n_{k} = n_{k-1}(n_{k-1}+1) > n_{k-2}(n_{k-2}+1)$.

If $n_{k-1} = n_{k-2} + 1$, then $n_{k-2}$ is even and, by the assumption $k \ge 4$, larger than $n_1 \ge 2$. Hence, $n_{k-2}$ is composite and we can write $n_{k-2} = ab$ with $a$ and $b$ integers larger than $1$. We then see $$m \in (n_1,\dots,n_{k-3},n_{k-1}, n_k, n_{k-2}+a, b(n_{k-2}+a)) \in S_{k+1},$$ where $$\{n_{k-1}, n_k\} \cap \{n_{k-2}+a,b(n_{k-2}+a) \} = \emptyset$$ follows similarly as before, in this case using the assumption  $n_{k-1} = n_{k-2} + 1 \neq n_{k-2} + a$.
\end{proof}

We now recall the following result by Vose~\cite{VoseBLMS}:

\begin{lem}[\cite{VoseBLMS}]\label{vose}
    There exists a positive integer $\alpha$, a sequence\footnote{Although it is not explicitly mentioned in~\cite{VoseBLMS} that we may assume all primes are larger than or equal to $5$, this does follow from the proof of~\cite[Lemma~3]{VoseJNT}. We also use ever so slightly different definitions of $N_K$, $u_i$ and $v_j$, but these differences are mainly cosmetic.} of primes $5 \le p_1 < p_2 < \cdots$ and a sequence of positive integers $N_1 < N_2 < \cdots$ defined by $$N_K :=  4^{\alpha K^2} (p_1p_2\cdots p_K)^2$$ for which the following holds: for all fractions $\frac{a}{b} \in (0, 1)$ there is an integer $K$ and divisors $$d_1, \ldots, d_r, d'_1, \ldots, d'_s$$ of $N_K$ such that with $u_i := \frac{N_K}{d_i}$ and $v_j := \frac{bN_K}{d'_j}$ we have both $$1 < u_1 < \cdots < u_r < v_1 < \cdots < v_s$$ and $$r+s \le C \sqrt{\log b}$$ for some absolute constant $C$, while $$\frac{a}{b} = \frac{1}{u_1} + \ldots + \frac{1}{u_r} + \frac{1}{v_1} + \ldots + \frac{1}{v_s}.$$
\end{lem}

The proof of Theorem~\ref{thm:main_vkeck2} will now be based on applying Lemma \ref{vose} to the fraction $\frac{a}{b} = \frac{m-1}{m}$.

\begin{proof}[Proof of Theorem~\ref{thm:main_vkeck2}]
By iterating Lemma \ref{lem:extension} it is sufficient to show that there exists an absolute constant $c > 0$ such that for all integers $k \geq 1$ and all integers $m$ with $1 < m < e^{ck^2}$ we have $m \in D_{k'}$ for some $k' \le k$. That is, we need to prove that there exists a unit fraction decomposition of $1$ with at most $k$ terms that contains $m$ as one of the denominators.

With $C$ as in Lemma~\ref{vose} we choose
\[
c := \min\left(\frac{\log 2}{432^2}, \frac{1}{(433C)^2}\right).
\]
Now let $m$ be any integer with $1 < m < e^{ck^2}$, and note that we may assume $k\ge 433$, as for $k\le 432$ there are no integers $m$ with $1 < m < e^{ck^2}$ by the definition of $c$.

We then first consider $m \le 432$, in which case it is by the assumption $k \ge 433$ sufficient to show that the denominator $m$ occurs in a unit fraction decomposition of $1$ with at most $m+1$ terms. If $m \le 432$ is not of the form $q(q+1)$ with $q \in \mathbb{N}$, then we have the decomposition
\[
1=\frac1m+\sum_{i=1}^{m-1}\frac{1}{i(i+1)}.
\]
The denominators in this decomposition are distinct, since the products $i(i+1)$ are strictly increasing and none of them are equal to $m$.

On the other hand, if $m=q(q+1)$ for some positive integer $q$, then we use
\[
1=\frac1m+\frac{1}{m+1}+\frac{1}{m(m+1)}
  +\sum_{\substack{1\le i\le m-1\\ i\ne q}}\frac{1}{i(i+1)}.
\]
We claim that this is again a decomposition into distinct unit fractions. Indeed, the products $i(i+1)$ are strictly increasing, the product $q(q+1)=m$ has been omitted, and neither $m+1$ nor $m(m+1)$ is equal to one of the remaining products, since
\[
q(q+1)<m+1<(q+1)(q+2),
\]
while $i(i+1)<m(m+1)$ for all $1\le i\le m-1$. Hence, for every $2\le m\le 432$, there is such a decomposition with at most $m+1 \le 433$ terms. For the remainder of the proof, we may therefore assume both $k\ge 433$ and $m \ge 433$. We then apply Lemma~\ref{vose} to the fraction $\frac{a}{b} = \frac{m-1}{m}$, and there are two different cases to consider.

If $m$ does not occur among the $u_i$ and $v_j$, then $$\frac{1}{u_1} + \ldots + \frac{1}{u_r} + \frac{1}{v_1} + \ldots + \frac{1}{v_s} + \frac{1}{m}$$ is a unit fraction decomposition of $1$ containing $m$ as a denominator. Moreover, the number of terms is
\begin{align*}
r+s+1 &\le C \sqrt{\log m} +1 \\
&< C\sqrt{c}k + 1 \\
&\le \frac{k}{433} + 1 \\
&< k.
\end{align*}

On the other hand, if $m$ does occur as one of the denominators $u_i$ or $v_j$, then let $t \in \{1, 2, 3\}$ be such that $m+t$ is divisible by $3$ and define $$D := \{3, 6, 9, 12, 15, 18, 24, 27, 30, 45, 54, 60, 72\}.$$ We note that all $d \in D$ are divisible by $3$ and one can verify that $\sum_{d \in D} \frac{1}{d} = 1$. We now consider
\begin{equation} \label{egypt}
\frac{1}{u_1} + \ldots + \frac{1}{u_r} + \frac{1}{v_1} + \ldots + \frac{1}{v_s} + \frac{1}{m+t} + \sum_{i=0}^{t-1} \sum_{d \in D} \frac{1}{d(m+i)(m+i+1)}.
\end{equation}
Since $\sum_{d \in D} \frac{1}{d} = 1$, we first note that
\begin{align*}
\frac{1}{m+t} + \sum_{i=0}^{t-1} \sum_{d \in D} \frac{1}{d(m+i)(m+i+1)}
&= \frac{1}{m+t} + \sum_{i=0}^{t-1} \frac{1}{(m+i)(m+i+1)} \\
&= \frac{1}{m+t} + \sum_{i=0}^{t-1} \left(\frac{1}{m+i} - \frac{1}{m+i+1} \right) \\
&= \frac{1}{m},
\end{align*}
as the sum telescopes. We therefore deduce that the total sum in \eqref{egypt} is equal to $1$. As before, the number of terms in \eqref{egypt} satisfies
\begin{align*}
r+s+1+t|D| &\le r+s+40 \\
&<  C\sqrt{c}k + 40 \\
&\le \frac{k}{433} + 40 \\
&\le k,
\end{align*}
by our choice of $c$ and the fact that $k\ge 433$. Therefore, to complete the proof it suffices to show that all denominators in \eqref{egypt} are distinct.

To see this, observe that every denominator in \eqref{egypt} other than the $u_i$ and $v_j$
is a multiple of $3$. On the other hand, recall that \(N_K=4^{\alpha K^2}(p_1p_2\cdots p_K)^2\) with $p_i \ge 5$ for all $i$, so that $3\nmid N_K$. Hence, none of these new denominators can divide $N_K$, and in particular none of them can coincide with any $u_i$.
Therefore, it remains to check that the denominators $m+t$ and $d(m+i)(m+i+1)$ are all
distinct from the $v_j$ (or equivalently, that none of them are equal to $m$ times a divisor of $N_K$),
and also distinct from one another.

Now, $m+t \in \{m+1,m+2,m+3 \}$ is not a multiple of $m$ if $m \ge 4$, while $m+t$ is smaller than all the other terms $d(m+i)(m+i+1)$, so the denominator $m+t$ in \eqref{egypt} is certainly unique.

If $i=0$, then $d(m+i)(m+i+1)=dm(m+1)$ cannot be equal to $m$ times a divisor of $N_K$, as $d(m+1)$ is divisible by $3$, whereas $3\nmid N_K$.

To see why $d(m+i)(m+i+1)$ is, for $1 \leq i \leq t-1$, not a multiple of $m$ and therefore not equal to any $v_j$, note that
\begin{align*}
\gcd\big(m, d(m+i)(m+i+1)\big) &\le \gcd(m, d)\gcd(m, m+i)\gcd(m, m+i+1) \\
&= \gcd(m, d)\gcd(m, i)\gcd(m, i+1)\\
&\le d(t-1)t\\
&\le 432 \\
&< m.
\end{align*}
We conclude that none of the $d(m+i)(m+i+1)$ are equal to a $v_j$.

Finally, we need to check that
\[
d(m+i)(m+i+1)\neq d'(m+j)(m+j+1)
\]
for all $d,d'\in D$ and all integers $0\le i,j \leq t-1$, unless $d=d'$ and $i=j$. Suppose, for a contradiction, that
\begin{equation}\label{eq:collision}
d(m+i)(m+i+1)=d'(m+j)(m+j+1)
\end{equation}
does hold for some $d,d'\in D$ and $0\le i,j \leq t-1$. By symmetry we may assume $d\ge d'$. In fact, if $d=d'$, then equation \eqref{eq:collision} simplifies to
\[
(m+i)(m+i+1)=(m+j)(m+j+1).
\]
As the function $x\mapsto (m+x)(m+x+1)$ is strictly increasing for $x\ge 0$, this would imply $i=j$. Hence, we will further assume the strict inequality $d > d'$.

Inspecting the set $D$, one can verify that for any distinct $d>d'$ in $D$ we actually have
\begin{equation} \label{eq:tenovernine}
\frac{d}{d'}\ge \frac{10}{9}.
\end{equation}
Since $t\in\{1,2,3\}$, we have $i,j\in\{0,1,2\}$, hence
\begin{equation} \label{eq:ismall}
d(m+i)(m+i+1)\ge dm(m+1)
\end{equation}
and
\begin{equation} \label{eq:jlarge}
d'(m+j)(m+j+1)\le d'(m+2)(m+3).
\end{equation}
By combining~\eqref{eq:collision},~\eqref{eq:tenovernine},~\eqref{eq:ismall}, and~\eqref{eq:jlarge}, we then obtain
\[
\frac{10}{9}m(m+1) \le (m+2)(m+3),
\]
contradicting $m > 432$. This completes the proof.
\end{proof}

\section{Concluding Remarks}\label{sec:concluding}
For integers $1\le a<b$, let $N(a,b)$ denote the minimal $t$ such that there exist distinct integers
$1<n_1<\cdots<n_t$ with
\[
\frac{a}{b}=\frac1{n_1}+\cdots+\frac1{n_t},
\]
and set $N(b):=\max_{1\le a<b} N(a,b)$.
The problem of estimating $N(b)$ goes back to Erd\H{o}s~\cite{Er50}, where he showed $$N(b) \ll \frac{\log b}{\log \log b}.$$ Erd\H{o}s and Graham~\cite[p.~37]{ErGr80} asked to improve upon this, and the determination of $N(b)$ is now recorded
as Erd\H{o}s Problem~\#304 on Bloom's Erd\H{o}s Problems website~\cite{EP304}. The current best known upper bound is by Vose~\cite{VoseBLMS} and gives
\begin{equation}\label{eq:Nb_vose}
N(b)\ll \sqrt{\log b}.
\end{equation}
The stronger conjecture that we actually have
\begin{equation}\label{eq:Nb_conj}
N(b)\ll \log\log b
\end{equation}
was first suggested in \cite{Er50}.

It is Vose's proof of \eqref{eq:Nb_vose} that we use in order to prove Theorem~\ref{thm:main_vkeck2}. If, however, \eqref{eq:Nb_conj} holds, it seems likely that our lower bound on $v(k)$ can be improved to $e^{e^{ck}}$ instead, which would match the doubly exponential growth rate suggested by Erd\H{o}s and Graham (up to the constant $c$ in the exponent). 

Conversely, any lower bound for $v(k)$ immediately yields an upper bound for $N(b-1,b)$. 
Indeed, $b<v(k)$ implies $b\in D_k$, so there is a $k$-term decomposition of $1$ involving $\frac{1}{b}$.
Removing this term gives a $(k-1)$-term decomposition of $\frac{b-1}{b}$, hence $N(b-1,b)\le k-1$.
This explains why Erd\H{o}s Problems \#293~\cite{EP} and \#304~\cite{EP304} are closely connected.

\section*{Acknowledgements}
The authors are grateful to Thomas Bloom for founding and maintaining the Erd\H{o}s Problems website. We further thank the anonymous referee for pointing out a simplification in our original treatment of the case of small $m$.

\end{document}